\documentclass[12pt,reqno]{amsart}
\usepackage{amsmath,amssymb,amsthm}
\allowdisplaybreaks[1]

\newtheorem{thm}{Theorem}[section]
\newtheorem{lemma}[thm]{Lemma}
\newtheorem{cor}[thm]{Corollary}

\theoremstyle{definition}

\newcommand{\thmref}[1]{Theorem~\ref{#1}}
\newcommand{\lemref}[1]{Lemma~\ref{#1}}

\newcommand{\Ad}{{\rm Ad} \hspace{0.1em}}
\newcommand{\Tr}{{\rm Tr} \hspace{0.1em}}
\newcommand{\supp}{{\rm supp} \hspace{0.1em}}
\newcommand{\spec}{{\rm sp} \hspace{0.1em}}

\renewcommand{\l}{\ell}

\providecommand{\norm}[1]{\left\lVert#1\right\rVert}

\title{Automorphisms of a non-type I $C^*$-algebra}
\author[A. Noguchi]{Akira Noguchi}
\address{
Department of Mathematics, Hokkaido University,
Hokkaido\\ \indent \mbox{060-0810},
JAPAN}
\email{anoguchi@math.sci.hokudai.ac.jp}
\begin{document}

\maketitle

\begin{abstract} 
Glimm's theorem says that a UHF algebra is almost embedded in a separable $C^*$-algebra not of type I. 
Applying his methods we obtain a covariant version of his result; 
a UHF algebra with a product type automorphism is covariantly embedded in such a $C^*$-algebra 
equipped with an automorphism with full Connes spectrum.
\end{abstract}

\section{Introduction}
Today it is important to study group actions on operator algebras, 
both of $C^*$-algebras and von Neumann algebras. 
In this paper we treat an embedding problem of automorphisms, using Glimm's idea.

In \cite{G}, Glimm studied type I $C^*$-algebras; 
a $C^*$-algebra $A$ is {\em of type I} 
if each non-zero quotient of $A$ contains a non-zero positive element $x$ such that $xAx$ is commutative. 
A part of his proof implies a celebrated theorem known as {\em Glimm's theorem}: 
For a separable $C^*$-algebra $A$ which is not of type I and a UHF algebra $D$, 
there is a $C^*$-subalgebra $B$ of $A$ and a closed projection $q$ in the enveloping von Neumann algebra of $A$ 
such that $q \in B' $, $qAq=Bq$ and $Bq \simeq D$, 
where $B' $ is the commutant of $B$. 
Roughly speaking, this theorem says that {\em any UHF algebra is almost embedded in such a $C^*$-algebra}. 
In fact, he proved this theorem only for $D = \otimes _{n=1}^{\infty } M_2 $, 
known as the Fermion algebra, 
and Pedersen arranged his proof and generalized to the case of an arbitrary UHF algebra in \cite{P}.
 
According to Glimm's theorem, we are able to embed UHF algebras. 
How about group actions? 
It is still an open problem whether or not general actions of UHF algebras can be embedded. 
Bratteli, Kishimoto and Robinson first succeeded in embedding actions of compact groups of a special type in \cite{BKR}. 
They embedded an action of a compact group on a UHF algebra $\otimes _{n=1}^{\infty } M_{k_n} $ 
of the form $\gamma _t = \otimes _{n=1}^{\infty } \Ad u_{nt} $, 
where  $t \mapsto u_{nt} $ is a unitary representation on $M_{k_n} $. 
They call an action of this form ''a product type action.'' 
Since any irreducible representation of a compact group is finite dimensional, 
a product type action seems standard. 
One decade and a half later, 
product type actions of $\mathbb{R}$ were embedded by Kishimoto in \cite{Ki2}. 
While $\mathbb{R}$ itself is easy to understand, 
non-compactness of $\mathbb{R}$ makes this embedding problem much more delicate, 
and the action (called "flow") need to be perturbed. 
In this paper, we treat the $\mathbb{Z}$-action case, 
i.e. the automorphism case. 
Since $\mathbb{Z}$ is not compact, 
a perturbation is also needed in this case. 
So the result is as follows:

\begin{thm} \label{main}
Let $A$ be a separable prime $C^*$-algebra and $\alpha $ an automorphism of $A$. 
Then the following are equivalent: 
\begin{enumerate}
\item 
The Connes spectrum $ \Gamma (\alpha ) $ of $\alpha $ is equal to $\mathbb{T} $. 
\label{fcs} 
\item 
For any UHF algebra $D =  \otimes _{n=1} ^{\infty} M_{k_n} $, 
any  automorphism $\gamma $ of $D$ of the form $\gamma = \otimes _{n=1} ^{\infty} \Ad e^{ih_n} $, 
where $M_{k_n} $ is the $k_n \times k_n $ matrix algebra with $k_n \geq 2 $ and $h_n \in M_{k_n} $ a self-adjoint matrix for each $n$, and any $\epsilon > 0$, 
there is a $C^*$-subalgebra $B$ of $A$, a unitary $v$ in $A$ (in $A+ \mathbb{C} 1 $ if $A$ is not unital) 
and a closed projection $q$ of the enveloping von Neumann algebra of $A$ which is in the commutant of $B$ 
such that
\begin{align*}
& \norm {v-1} < \epsilon , \quad 
\alpha ^{(v)} (B) = B , \\
& ( \alpha ^{(v)} )^{**} (q) = q , \quad 
qAq = Bq , \\
& (Bq , ( \alpha ^{(v)} )^{**} |Bq )  \simeq  (D, \gamma ) 
\end{align*} 
and for $x \in A$, $x=0$ if and only if $xc(q)=0$, 
where $\alpha ^{(v)} := \alpha \ \circ $ $\Ad v$ is a perturbation of $\alpha $ 
and $c(q)$ is the central cover of $q$ 
and $(\alpha ^{(v)} )^{**} |Bq $ is the restriction of $(\alpha ^{(v)} )^{**} $ to $Bq$. 
\label{qma}
\end{enumerate}
\end{thm}

In the statement above, 
the $\sigma $-weakly extended automorphism of $\alpha ^{(v)} $ to the enveloping von Neumann algebra of $A$ 
is denoted by $( \alpha ^{(v)} )^{**} $, 
but we will later omit the stars; 
the same applies to representations, etc.

It seems natural that condition (\ref{fcs}) is necessary when the condition (\ref{qma}) is true. 
If $A$ was simple and $ \Gamma (\alpha ) \neq \mathbb{T} $, 
$\alpha ^n$ would be inner in the multiplier algebra of $A$ for some $n$ (8.9.9 in \cite{Ped}), 
so very few $\gamma $'s would satisfy \thmref{main}. 

In the hypothesis of the theorem above, 
if $A$ has a faithful irreducible representation and $ \Gamma (\alpha ) = \mathbb{T} $, then $A$ is automatically not of type I. 
This can be proved as follows. 
Suppose that $x$ is a positive element of $A$ such that $xAx$ is commutative. 
The norm closure of $xAx$ is a hereditary sub-$C^*$-algebra of $A$, 
whose image of irreducible representation is an algebra of one-dimensional operators. 
This contradicts $ \Gamma (\alpha ) = \mathbb{T} $ (by the same argument as in the proof of \lemref{ncpt}). 

Note that the $\sigma $-weak closure of a UHF algebra 
can be an AFD (approximately finite dimensional) factor of various type; 
concretely, of type II$_1$, II$_{\infty} $ and III$_{\lambda } $, $0 \leq \lambda \leq 1$. 
Here is an example of construction of an AFD factor of type III$_{\lambda } $, $0 < \lambda < 1$. 
Let $\phi ^{(2)} $ be the $\Ad \left( \begin{array}{cc} 1 & 0 \\ 0 & \lambda ^{it} \end{array} \right) $-KMS-state on $M_2 $; i.e. 
\[ \phi ^{(2)} (x):=\frac{\Tr \left( \left(
\begin{array}{cc} 
1 & 0 \\
0 & \lambda \\
\end{array}
\right) x \right)}
{\Tr \left(
\left( \begin{array}{cc} 
1 & 0 \\
0 & \lambda \\
\end{array}
\right) \right) } \]
for $x \in M_2 $, where $\Tr $ denotes the usual trace on $M_2 $, 
and set $\phi := \otimes _{n=1}^{\infty} \phi _n $, 
where $\phi _n := \phi ^{(2)} $ for each $n$. 
Then it follows that the $\sigma $-weak closure of $\pi _{\phi } ( \otimes _{n=1}^{\infty } M_2 )$, 
denoted by $\pi _{\phi } ( \otimes _{n=1}^{\infty } M_2 )''$, 
where $\pi _{\phi } $ is the GNS representation of $\phi $, 
is an AFD factor of type III$_{\lambda } $ (see \cite{T}, XVIII.1.1).
We state a straightforward corollary and end the introduction.

\begin{cor} 
Let $A$ be a separable prime $C^*$-algebra 
and $\alpha $ an automorphism of $A$ with the Connes spectrum $ \Gamma (\alpha ) = \mathbb{T} $. 
Then, for any AFD factor $M$, 
there are an $\alpha $-covariant representation $\pi $ of $A$ and a projection $Q$ of $\pi (A)'' $ with $c(Q)=1$ 
such that $Q\pi (A)'' Q \simeq M$.
\end{cor}

\begin{proof}
Note that an AFD factor always has a $\sigma$-weakly dense UHF subalgebra $D$ (\cite{EW}). 
We use \thmref{main} for $\gamma =$identity and obtain $B$, $v$ and $q$. 
We take a faithful state on $M$ and restrict it on $D$. 
This state gives one on $qAq=Bq$ through the isomorphism $(Bq , \alpha ^{(v)} |Bq )  \simeq  (D, \gamma )$, 
which is denoted by $\psi _0$. 
Because of a choice of $\gamma$, $\psi _0$ is $\alpha ^{(v)} |Bq$-invariant. 
We define a state $\psi$ on $A$ by $\psi (x):=\psi _0 (qxq) $ for $x \in A$. 
Let $(\pi _{\psi } , {\mathcal H}_{\psi } , \xi _{\psi })$, $(\pi _{\psi_0 } , {\mathcal H}_{\psi_0 } , \xi _{\psi_0 })$ 
be the GNS-triples of $\psi $ and $\psi _0 $, resp. 
Set $Q:=\pi _{\psi } (q)$. 
Then it follows that $Q\pi_{\psi}(A)Q=\pi_{\psi} (qAq)$, 
which implies $Q\pi_{\psi}(A)''Q=\pi_{\psi} (qAq)''$. 
We check that there is a natural isomorphism 
between $\overline{\pi_{\psi_0} (qAq)''\xi_{\psi_0}}^{\norm{\cdot}}$ and $\overline{\pi_{\psi} (qAq)''\xi_{\psi}}^{\norm{\cdot}}$. 
For any $x, y \in A$, there are $z,w \in B$ such that $qxq=zq$ and $qyq=wq$. 
So we have 
\begin{align*} \langle \pi_{\psi} (qxq)\xi _{\psi} , \pi_{\psi} (qyq)\xi _{\psi} \rangle 
&= \langle \pi_{\psi} (zq)\xi _{\psi} , \pi_{\psi} (wq)\xi _{\psi} \rangle \\
&= \langle \pi_{\psi} (z)\xi _{\psi} , \pi_{\psi} (w)\xi _{\psi} \rangle \\
&= \psi (w^* z) 
= \psi _0 (qw^* zq) \\
&= \langle \psi _0 (zq)\xi _{\psi _0} , \psi _0 (wq)\xi _{\psi _0} \rangle \\
&=\langle \psi _0 (qxq)\xi _{\psi _0} , \psi _0 (qyq)\xi _{\psi _0} \rangle , \\ \end{align*} 
since $\psi (q)=1$ implies $\pi_{\psi} (q)\xi _{\psi} =\xi _{\psi} $, where $ \langle \cdot , \cdot \rangle $ denotes the inner product.
Thus we may assume that 
${\mathcal H}_{\psi_0 } \subset {\mathcal H}_{\psi}$ and $\pi _{\psi }$ is an extension of $\pi _{\psi_0 }$.
By the construction of $\psi _0$, 
it follows that $\pi _{\psi _0} (qAq)'' \simeq M$, 
whence $Q\pi (A)''Q \simeq M$. 
Finally, since 
\begin{align*} 
c(Q) {\mathcal H} _{\psi} 
&= \overline{\pi _{\psi} (A) Q \pi _{\psi } (A) \xi _{\psi}}^{\norm{\cdot}} \\
& \supset \overline{\pi _{\psi} (A) Q \xi _{\psi}}^{\norm{\cdot}} 
= \overline{\pi _{\psi} (A) \xi _{\psi}}^{\norm{\cdot}} 
={\mathcal H} _{\psi } , 
\end{align*} 
we have $c(Q)=1$.
\end{proof}
 
\noindent 
{\bf Notations.} 
For a Hilbert space $\mathcal{H}$, 
$ \langle \cdot , \cdot \rangle $ denotes the inner product of $\mathcal{H}$, 
${\bf B} ({\mathcal H})$ the set of bounded operators on ${\mathcal H}$, 
and ${\bf K} ({\mathcal H})$ the set of compact operators on ${\mathcal H}$. 
For a $C^*$-algebra $A$, 
$A_{sa}$ denotes the set of self-adjoint elements in $A$, 
$A_{+}$ the set of positive elements in $A$, $A_1 $ the unit ball of $A$, 
and ${\mathcal U} (A)$ the set of unitary elements in $A$ (in $A+ \mathbb{C} 1 $ if $A$ is not unital). 
We denote by $A^{**} $ the enveloping von Neumann algebra of $A$. 
When $A$ is in some von Neumann algebra, 
$A'$ denotes the commutant of $A$ 
and $A''$ the double commutant of $A$, 
which is equal to the $\sigma $-weak closure of $A$. 
For an automorphism $\alpha $ of $A$, $\hat{\alpha} $ denotes the dual action of $\alpha $. 
For a unitary $U$ in ${\bf B} ({\mathcal H})$, 
$E_U$ denotes the spectral measure (on $\mathbb{T}$) of $U$. 
For a state $\phi $ of a $C^*$-algebra $A$, $\pi _{\phi} $ denotes the GNS representation of $\phi $, 
and $\supp \phi \in A^{**} $ the support projection of $\phi $. 
For a function $f$, $\supp f$ denotes the support of $f$.

\vspace{1em}

\noindent 
{\bf Acknowledgement.} 
The author is grateful to Akitaka Kishimoto for improving the contents and pointing out errors. 
The author is also indebted to Reiji Tomatsu for some pieces of advice.
%%%%%%%%%%%%%%%%%%%%%%%%%%%%%%%%%%%%%%%%%%%%%%%%%%%%%%%%%%%%%%%%%%%%%%%%%%
\section{Proof of the main theorem} 
We can prove that (\ref{qma}) implies (\ref{fcs}) easily, so we prove it first. 
(We put stars for $\sigma $-weakly continuous extensions of automorphisms only in this proof.)
Let $D:= \otimes _{n=1}^{\infty } M_{k_n} $, where $k_n := 2$ for each $n$. 
Set $u_n :=\left( \begin{array}{cc} 1 & 0 \\ 0 & e^{2\pi i \theta } \end{array} \right) $ for each $n$, 
where $\theta $ is an arbitrary irrational number independent on $n$, 
and define an automorphism of $D$ by $\gamma :=\otimes \Ad u_n $. 
We get an isomorphism $(qAq , (\alpha ^{(v)})^{**} |qAq )  \simeq  (D, \gamma )$, 
where $q$ and $v$ are obtained by the condition (\ref{qma}). 
Let $\tau $ is the tracial state on $qAq$. 
Then it follows that $\pi _{\tau } (qAq)'' $ is the hyperfinite II$_1$-factor. 
Since $\sum _{n=1}^{\infty } |1-|(1+e^{2m\pi i \theta })/2||= \infty $,
the $\sigma $-weakly continuous extension of $(\alpha ^{(v)} )^m $ to $\pi _{\tau } (qAq)'' $ is outer for any $m \in \mathbb{Z} \backslash \{ 0 \} $ (\cite{C}, 1.3.7). 
Define a state $\psi $ on $A$ by $\psi (x) := \tau (qxq) $ for $x \in A$. 
Then, since $\psi \circ (\alpha ^{(v)})^{**} =\psi $, 
we can extend $(\alpha ^{(v)} )^{**}$ to $\pi _{\psi } (A)''$. 
Since $(\alpha ^{(v)})^{**} (\pi _{\psi} (q))=\pi _{\psi} (q)$, 
it follows that $((\alpha ^{(v)})^{**} )^m $ on $\pi _{\psi } (A)''$ is also outer for each $m$, 
whence $\Gamma (\alpha ^{**})=\mathbb{T}$ since $\mathbb{Z}$ is discrete.
Therefore we have $\Gamma (\alpha )=\mathbb{T}$ (see \cite{Ped}, 8.8.9). 

We will show that (\ref{fcs}) implies (\ref{qma}) from now. 
Before we begin the proof, 
we present Kadison's transitivity in the following form.
 
\begin{lemma} \label{kad}
For any $\epsilon >0$ and any natural number $m$, 
there is a $\delta >0$ such that the following holds:

Let $A$ be a $C^*$-algebra, 
$\pi $ an irreducible representation on a Hilbert space ${\mathcal H}$, 
and $V$ a unitary in ${\bf B} ({\mathcal H})$. 
Let $(\xi _1 , \cdots , \xi _m )$ be a finite family of mutually orthogonal unit vectors in ${\mathcal H}$. 
If $\norm{V\xi _j - \xi _j } < \delta $ for $j=1, \cdots , m$, 
there is a $v$ in ${\mathcal U} (A)$ such that 
$\norm{v-1} < \epsilon $ and $\pi (v) \xi _j = V \xi _j $ for $j=1, \cdots , m$. 
\end{lemma}

To prove this, we prepare the following lemma.

\begin{lemma} \label{gs}
Let $m,n$ be natural numbers and $\epsilon >0$. 
Let $(\xi _1 , \cdots , \xi _m )$ and $(\eta _1 , \cdots , \eta _n )$ be two families of unit vectors 
such that $\xi _j $'s are mutually orthogonal and 
$|\langle \eta _i , \xi _j \rangle | < \epsilon $, $|\langle \eta_i , \eta _k \rangle | < \epsilon $ 
for any $j=1, \cdots m$ and $i,k=1, \cdots , n$, $i \neq k$. 
Then there is a family $(\eta _1 ' , \cdots , \eta _n ' )$ of unit vectors 
in the finite dimensional subspace spanned by $\xi _1 , \cdots , \xi _m , \eta _1 , \cdots , \eta _n $ 
such that $(\xi _1 , \cdots , \xi _m , \eta _1 ' , \cdots , \eta _n ' )$ is an orthogonal family of unit vectors 
and $\norm{\eta _i - \eta _i ' } <r_{mn} \epsilon $ for $i=1, \cdots , n$, 
where $r_{mn} $ is a positive real number dependent on $m$ and $n$. 
\end{lemma}

\begin{proof}
We recall the process of the Gram-Schmidt orthogonalization. 
Define $\eta _1 '' := \eta _1 - \sum _{j=1}^m \langle \eta _1 , \xi _j \rangle \xi _j $. 
Then we have $\langle \eta _1 '' , \xi _j \rangle =0$ for $j=1, \cdots , m$ and 
\[ \norm{\eta _1 '' - \eta _1 } \leq \sum _{j=1}^m |\langle \eta _1 , \xi _j \rangle | < m \epsilon . \] 
And define $\eta _1 ':= \eta _1 '' / \norm{\eta _1 '' } $. 
Since $1-m \epsilon < \norm{\eta _1 '' } \leq 1 $, 
we have 
\begin{align*} 
\norm{\eta _1 ' - \eta _1 } 
& \leq \norm{\eta _1 ' - \eta _1 '' } + \norm{\eta _1 '' - \eta _1 } \\ 
& \leq (1- \norm{\eta _1 '' } ) \norm{\eta _1 ' }  + \norm{\eta _1 '' - \eta _1 } \\ 
&< m \epsilon + m \epsilon = 2m \epsilon . 
\end{align*} 
When $\eta _1 ', \cdots , \eta _{i-1} ' $ have already defined for $2 \leq i \leq n$, 
set $\eta _i '' := \eta _i - \sum _{j=1}^m \langle \eta _i , \xi _j \rangle \xi _j - \sum _{\l=1}^{i-1} \langle \eta _i , \eta _{\l} ' \rangle \eta _{\l} ' $ and $\eta _i ':= \eta _i '' / \norm{\eta _i '' } $. 
As above, it follows that $\langle \eta _i ' , \xi _j \rangle =0$ for $j=1, \cdots , m$, 
$\langle \eta _i ' , \eta _{\l} ' \rangle =0$ for $\l=1, \cdots , i-1 $, 
and $\norm{\eta _i ' - \eta _i } < r_i \epsilon $ for all $i$, 
when we set $r_1 :=2m$ and $r_i :=2 (m + i-1 + r_1 + r_2 + \cdots + r_{i-1} )$ for $2 \leq i \leq n$. 
Since the sequence $(r_i )_i $ is obviously increasing, 
we have $r_i = 2(m + i-1 + r_1 + \cdots + r_{i-1} ) \leq 2n r_{i-1} $, 
whence $r_i \leq 2m(2n)^{n-1} $ for $1 \leq i \leq n$.
\end{proof}

\noindent 
\begin{proof}[Proof of \lemref{kad}]

We may assume that $\epsilon < 1/2 $. 
Let ${\mathcal F}$ be the finite-dimensional subspace of ${\mathcal H}$ 
spanned by $\xi _1 , \cdots , \xi _m $ and $V \xi _1 , \cdots , V \xi _m $. 
Let $\eta _1 , \cdots , \eta _n $ be unit vectors 
such that $(\xi _1 , \cdots , \xi _m , \eta _1 , \cdots , \eta _n )$ is an orthonormal basis of ${\mathcal F}$. 
Since 
\[ 
|\langle \eta _i , V \xi _j \rangle | 
= |\langle \eta _i , V \xi _j \rangle - \langle \eta _i , \xi _j \rangle | \leq \norm{V \xi _j - \xi _j } 
< \delta  
\] 
for $i=1, \cdots , n$ and $j=1 , \cdots , m$, 
it follows from \lemref{gs} that there is a family $(\eta _1 ' , \cdots , \eta _n ' )$ of unit vectors in ${\mathcal F}$ 
such that $(V \xi _1 , \cdots , V \xi _m , \eta _1 ' , \cdots , \eta _n ' )$ 
is an orthonormal basis of ${\mathcal F}$ 
and $\norm{\eta _i - \eta _i ' } <r_{mn} \epsilon $ for $i=1, \cdots , n$, 
where $r_{mn} $ is a positive real number dependent on $m$ and $n$. 
Let $W$ be a unitary on ${\mathcal F}$ 
determined by $W \xi _j := V \xi _j $ for $j =1, \cdots , m$ and $W \eta _i := \eta _i ' $ for $i=1, \cdots , n$. 
When we set $\delta := \epsilon / (2\sqrt{n} r_{mn} )$, 
it follows that $\norm{W-1} < \epsilon /2 $. 
Define $T:=-i \log W = i \sum _{n=1}^{\infty} (W-1)^n /n$ on ${\mathcal F} $. 
Then we have $\norm{T} < -\log (1-\epsilon /2)$. 
We extend $T$ to a self-adjoint operator on ${\mathcal H}$ 
by setting $T=0$ on the orthogonal complement of ${\mathcal F}$. 
We also denote this extended operator by $T$ and define $W=e^{iT} $ on ${\mathcal H} $. 
Let $P$ be the projection onto ${\mathcal F} $. 
By Kadison's transitivity for a self-adjoint operator, 
there is an $a \in A_{sa} $ such that $TP = \pi (a) P$ and $\norm{a} < - \log (1- \epsilon /2 )$. 
By the construction of $T$, we have $TP=PTP= \pi (a)P = P \pi (a) P$. 
Hence it follows that 
\begin{align*} 
WP=PWP 
&= P \sum _{n=0}^{\infty} \frac{(iT)^n}{n!} P
= \sum _{n=0}^{\infty} \frac{(iPTP)^n}{n!} \\ 
&= \sum _{n=0}^{\infty} \frac{(iP \pi (a) P)^n}{n!} 
=P \sum _{n=0}^{\infty} \frac{(i \pi (a))^n}{n!} P \\ 
&= P e^{i \pi (a)} P = e^{i \pi (a)} P  
\end{align*} 
and 
\[ 
\norm{e^{ia} -1} \leq e^{\norm{a}} -1 
< e^{- \log (1- \epsilon /2) } -1 
< \epsilon . 
\] 
Now $v := e^{ia} $ is a desired unitary.
\end{proof}

From now on, 
Let $\pi$ be a faithful $\alpha$-covariant irreducible representation of $A$ on a Hilbert space ${\mathcal H}$ 
and $U$ the implementing unitary of $\alpha$. 
The existence of such a $\pi $ is proved in \cite{BEEK}. 
Note that every pair of a $C^*$-algebra and its automorphism 
does not have a faithful covariant irreducible representation 
in the case where $\Gamma ({\alpha}) \neq \mathbb{T}$ 
(the definition of the Connes spectrum can be seen in \cite{Ped}, 8.8.2). 
Here is an example. 
Let $A_{\theta } $ be the irrational rotation algebra for an irrational number $\theta $, 
i.e. $A_{\theta } $ is the universal $C^*$-algebra generated by two unitaries $u$ and $v$ 
which satisfy the relation $uv=e^{2\pi i \theta } vu$, 
and $\alpha $ the automorphism of $A_{\theta } $ defined by 
\[ \alpha (u) := -u , \quad \alpha (v) := e^{\pi i \theta } v . \] 
Suppose that $A_{\theta } $ has a faithful irreducible representation $\sigma $ 
which satisfy $\Ad U \circ \sigma = \sigma \circ \alpha $, 
where $U$ is the implementing unitary. 
Since 
$\Ad U^2 \circ \sigma (u) 
= \sigma \circ \alpha ^2 (u) 
= \sigma (u) = (\Ad \sigma (u) \circ \sigma)(u)$ 
and 
$\Ad U^2 \circ \sigma (v) 
= \sigma \circ \alpha ^2 (v) 
= \sigma (e^{2\pi i \theta } v) 
= \sigma (uvu^* ) 
= (\Ad \sigma (u) \circ \sigma )(v) $, 
it follows that $\Ad U^2 = \Ad \sigma (u)$. 
Thus $U^2 \sigma (u)^* $ is in the commutant of $\sigma (A_{\theta } )$, 
which is equal to $\mathbb{C} $ since $\sigma $ is irreducible. 
We take a $\lambda \in \mathbb{C} $ so that $U^2 = \lambda \sigma (u)$. 
Then we have 
\[ 
\lambda \sigma (u) = U^2 = UU^2 U^* = U\lambda \sigma (u) U^* 
= \lambda (\sigma \circ \alpha )(u) = - \lambda \sigma (u) , 
\] 
which is absurd. 
 
Note that an $\hat{\alpha }$-invariant ideal of $A \rtimes _{\alpha} \mathbb{Z} $
induces a non-trivial $\alpha $-invariant ideal of $A$ 
by $y \mapsto I(y):= \int _{\mathbb{T}} \hat{\alpha _t} (y) dt $ for $y$ in the $\hat{\alpha }$-invariant ideal, 
where this integral converges in the norm topology since $\mathbb{T}$ is compact 
(see the proof of \cite{Ped}, 7.9.6).
 
\begin{lemma} \label{ncpt}
$(\pi \rtimes U )(A \rtimes _{\alpha} \mathbb{Z} )$ has no non-zero compact operators, 
where $\pi \rtimes U \colon A \rtimes _{\alpha} \mathbb{Z} \rightarrow {\bf B} ({\mathcal H})$ 
is a homomorphism 
defined by $\pi \rtimes U (y):=\sum _{n \in \mathbb{Z}} \pi (y(n))U^n $ for $y \in C_0 (\mathbb{Z} , A)$. 
\end{lemma}
 
\begin{proof}
At first, we show that $\pi (A)$ has no non-zero compact operators. 
We may identify $A$ with $\pi (A)$ 
and assume that $A$ is an irreducible subalgebra of ${\bf B} ({\mathcal H})$. 
Suppose that $A$ has a non-zero compact operator. 
Since $A$ is irreducible, $A$ contains ${\bf K} ({\mathcal H})$. 
It is obvious that $\alpha ({\bf K} ({\mathcal H}))= {\bf K} ({\mathcal H}) $. 
But, since $\alpha |_{{\bf K} ({\mathcal H})} $ is inner in ${\bf B} ({\mathcal H})$ 
(see the proof of \cite{Ped}, 8.7.4), 
the Connes spectrum of $\alpha |_{{\bf K} ({\mathcal H})} $ is equal to $\{ 0 \} $ 
(see \cite{Ped}, 8.9.10; note that the multiplier algebra of ${\bf K} ({\mathcal H}) $ is ${\bf B} ({\mathcal H})$). 
This contradicts $ \Gamma (\alpha ) = \mathbb{T} $.
 
Next we show that $(\pi \rtimes U )(A \rtimes _{\alpha} \mathbb{Z} )$ has no non-zero compact operators. 
If $(\pi \rtimes U )(A \rtimes _{\alpha} \mathbb{Z} )$ has a non-zero compact operator $K:= (\pi \rtimes U) (K') \geq 0$, 
then $\pi (I(K'))$ is a non-zero compact operator in $\pi (A)$, 
which contradicts the last paragraph.
\end{proof}
 
For an element $u$ in ${\mathcal U} (A)$, we define 
\[ U^{(u)} := U \pi (u) . \] 
Then it follows that $\Ad U^{(u)} \circ \pi = \pi \circ \alpha ^{(u)}$.
 
Note that since $ \Gamma (\alpha ) = \mathbb{T} $, 
it follows that $\spec (U)= \mathbb{T} $, 
where $\spec (U)$ is the spectrum of $U$, 
and $\pi \rtimes U$ is faithful.

\begin{lemma}
For any $\epsilon >0$, 
there are a $u$ in ${\mathcal U} (A)$ and a unit vector $\xi _0 $ in ${\mathcal H}$ 
such that $\norm{u -1} < \epsilon $ and $U^{(u)} \xi _0 = \xi _0 $.
\end{lemma}

\begin{proof}
Using the functional calculus, 
there is an $H$ in ${\bf B} ({\mathcal H})_{sa} $ such that $U = e^{iH}$. 
Let $\delta >0$. 
Applying Weyl's theorem, 
there is a compact operator $K$ in ${\bf B} ({\mathcal H})_{sa} $ 
such that $\norm{K} < \delta $ and $H-K$ is diagonal. 
Since 
\[ \frac{d}{ds} (e^{-isH} e^{is(H-K)}) = -e^{-isH} iK e^{is(H-K)} , \] 
we have 
\[ V := e^{-iH} e^{i(H-K)} = - \int_0^1 e^{-isH} iK e^{is(H-K)} ds +1. \] 
Then it follows that $\norm{V-1} \leq \norm{K} < \delta $. 
Since $UV$ is diagonal and $\spec (U)= \mathbb{T} $, 
there are a $\lambda \in \mathbb{T} $ and a unit vector $\xi _0 \in {\mathcal H}$ 
such that $\norm{\lambda -1} < \delta $ and $UV \xi _0 = \lambda \xi _0 $. 
Thus we have 
\[ 
\norm{U^* \xi _0 - \xi _0 } \leq \norm{U^* \xi _0 - V \xi _0 } + \norm{V \xi _0 - \xi _0 } 
\leq |\lambda -1| + \norm{V-1 } < 2 \delta . 
\] 
Now we can find a desired unitary $u$ by \lemref{kad}. 
\end{proof}
 
According to this lemma, 
we may assume that there is a unit vector $\xi _0$ in ${\mathcal H}$ 
such that $U \xi _0 = \xi _0$. 
Let $\omega$ be the pure state 
defined by $\omega (x) := \langle \pi (x) \xi _0 , \xi _0 \rangle $ for $x$ in $A$.
 
We define 
\[ T:= \{ e \in A | 0 \leq e \leq 1, \pi (e) \xi _0 = \xi _0 , \ {\rm and} \ \exists a \in A : ea=a , \pi (a) \xi _0 = \xi _0 \} . \] 

Note that we can always take an $a$ from $T$ in this definition. 
In fact, for $e \in T$ and $a \in A$ such that $ea=a$ and $\pi (a) \xi _0 = \xi _0 $, 
it follows that $ef(a)=f(a)$ and $\pi (f(a)) \xi _0 = \xi _0 $, 
where $f(t) :=2t \ (0 \leq t \leq 1/2)$, $:=1 \ (1/2 \leq t \leq 1)$. 
It is obvious that $f(a) \in T$.

\begin{lemma} \label{dec}
There is a decreasing sequence $(e_N ) _N $ in $T$ 
such that $e_N e_{N+1} = e_{N+1} $ for any $N = 1,2, \cdots$ 
and $e_N \searrow \supp \omega $, i.e. $\supp \omega = \inf _N e_N $. 
\end{lemma}

\begin{proof} 
Since $p := \supp \omega$ is a closed projection (see \cite{Ped}, 3.13.6), 
there is a decreasing sequence $( y_n )_n $ in the unit ball of $A_{+}$ such that $y_n \searrow p$. 
Put $y:= \sum_{n=1}^{\infty} 2^{-n} y_n$, 
which is in the unit ball of $A_{+}$. 
Then, for any state $\psi$ on $A$, 
it follows that $\psi (y)=1$ if and only if $\psi (p) =1$. 
This implies that for $\eta \in H$, 
it follows that $\pi (y) \eta = \eta $ if and only if $\eta \in \mathbb{C} \xi _0 $. 
Thus the spectral projection of $y$ (in $A^{**} $) corresponding to the eigenvalue $1$ is $p$. 
We define a sequence of continuous functions on $[0,1]$ by 
\[
f_N(t):=\left\{
\begin{array}{ll}
0 & \: ( 0 \leq t \leq 1-\frac{1}{2^N} ) \\
2^{N+1} t - 2^N & \: ( 1-\frac{1}{2^N} \leq t \leq 1-\frac{1}{2^{N+1}} ) \\
1 & \: ( 1-\frac{1}{2^{N+1}} \leq t \leq 1)
\end{array}
\right. . 
\] 
and set $e_N := f_N (y)$. 
Then $(e_N )_N $ is a decreasing sequence 
whose infimum is the spectral projection of  $y$ corresponding to the eigenvalue $1$, which is $p$. 
Since $\pi (p) \xi _0 = \xi _0 $, $y \geq p$ and $f_N (1)=1$, 
we have $\pi (e_N ) \xi _0 = \xi _0 $, whence $e_N \in T$. 
\end{proof}

Since $\omega (\alpha (p) ) = \omega (p) =1$, 
it follows that $\alpha (p) \geq p $. 
Taking $\alpha ^{-1} $ instead of $\alpha$, 
we have that $\alpha (p) =p$. 
 
Note that for an arbitrary positive element $x$ in $T$ such that $x \geq p $, 
this decreasing sequence can be taken so that $x \geq e_N $ for each $N$. 
We will check it. 
Since a state of a hereditary subalgebra extends uniquely to a state of the whole algebra (see \cite{Ped}, 3.1.6), 
the restriction of $\omega $ to the hereditary subalgebra $B := \{ y \in A | xy = yx =y \} $ of $A$ is also pure. 
Thus we can take the sequence $(y_n )_n $ from $B$ in the argument above. 
Then we have $e_N \leq x$.
 
\begin{lemma} \label{en}
If $f \in {\l}^1 (\mathbb{Z} ) $ satisfies that $f \geq 0$ and $\norm{f}_{\l ^1 (\mathbb{Z})} =1$, 
it follows that 
\begin{align*} 
& \lim_{M \to \infty} \norm{ \alpha _f (e_N) e_M - e_M } = 0 , \\ 
& \lim_{M \to \infty} \norm{ e_N \alpha _f (e_M) - \alpha _f (e_M) } = 0 
\end{align*} 
for each $N$, 
where we define $\alpha _f (x) := \sum _{n=- \infty} ^{\infty} f(n) \alpha ^n (x)$ 
for $f \in {\l}^1 (\mathbb{Z} ) $ and $x \in A$. 
\end{lemma}

\begin{proof}
Suppose that the first equality is not valid. 
Then there is a $\delta > 0$ such that there are infinitely many $M$'s which satisfy 
\[ \norm{ ( \alpha _f (e_N) -1) e_M^2 ( \alpha _f (e_N) -1) } > \delta . \] 
Since $(e_M^2 )_M $ is decreasing 
(because $(e_M )_M $ is decreasing and $e_M e_{M+1} =e_{M+1} $ for any $M$), 
this inequality holds for every $M$. 
We can take a state $\phi _M $ on $A$ such that 
\[ \phi _M ( ( \alpha _f (e_N) -1) e_M^2 ( \alpha _f (e_N) -1) ) > \delta  \] 
for every $M$. 
Since $(e_M^2 )_M $ is decreasing, we have 
\[ \phi _{M'} ( ( \alpha _f (e_N) -1) e_M^2 ( \alpha _f (e_N) -1) ) > \delta  \] 
for any $M' >M$. 
Taking a cluster point, we can find a state $\phi$ on $A$ such that 
\[ \phi ( ( \alpha _f (e_N) -1) e_M^2 ( \alpha _f (e_N) -1) ) \geq \delta  \] 
for any $M$, whence 
\[ \phi ( ( \alpha _f (e_N) -1) p ( \alpha _f (e_N) -1) ) \geq \delta , \] 
where $p := \supp \omega$. 
On the other hand, since $\alpha (p) = p$ and $e_N p = p$, we have 
\[ ( \alpha _f (e_N) -1) p = \sum _{n=- \infty} ^{\infty} f(n) (\alpha ^n(e_N p) -p) =0 , \] 
which is a contradiction. 
The second equality follows similarly.
\end{proof}

\begin{lemma} \label{eigqq} 
It follows that
\[ \norm{ \pi (e_N) E_U (q- \epsilon , q+ \epsilon )} =1 \]
for any $q$ in $\mathbb{T}$, $\epsilon > 0$ and $N = 1,2, \cdots$. 
\end{lemma}
 
\begin{proof}
Let $\lambda$ denote the canonical embedding of $C^* (\mathbb{Z} ) $ 
into the multiplier algebra $M(A \rtimes _{\alpha} \mathbb{Z} )$. 
For any $g \in C^* (\mathbb{Z} ) $, 
since $( \norm{e_N \lambda (g) e_N } )_N  $ is a decreasing sequence, 
\[ \rho (g) := \lim_{N \to \infty} \norm{e_N \lambda (g) e_N } \] 
exists. 
We will show that $\rho$ is a $C^*$-norm on $C^* (\mathbb{Z} ) $, 
whence $\rho (g) = \norm{g} $ for $g \in C^* (\mathbb{Z} ) $ 
because a $C^*$-norm on a $C^*$-algebra is unique.
 
For any $g \in C^* (\mathbb{Z})$ and any $f \in {\l}^1 (\mathbb{Z} ) $ 
such that $f \geq 0$ and $\norm{f}_{\l ^1 (\mathbb{Z})} =1$, 
since, for any $N$, 
\begin{align*} 
\lim _M \norm{e_M \lambda (g) e_M } 
&= \lim _M \norm{e_M \alpha _f (e_N) \lambda (g) \alpha _f (e_N) e_M } \\ 
&\leq \norm{\alpha _f (e_N) \lambda (g) \alpha _f (e_N) } 
\end{align*} 
by \lemref{en}, it follows that 
$\rho (g) \leq \lim _N \norm{ \alpha _f (e_N) \lambda (g) \alpha _f (e_N) } $. 
We can prove $\rho (g) \geq \lim _N \norm{ \alpha _f (e_N) \lambda (g) \alpha _f (e_N) } $ similarly, 
so we have 
\[ \rho (g) = \lim_{N \to \infty} \norm{ \alpha _f (e_N) \lambda (g) \alpha _f (e_N) } \] 
for $g$ in $C^* (\mathbb{Z})$ and $f \in {\l}^1 (\mathbb{Z} ) $ 
such that $f \geq 0$ and $\sum _{n=- \infty} ^{\infty} f(n) =1$. 
For any $g,h$ in $C^* (\mathbb{Z})$ and $\epsilon >0$, 
there is an $f$ in ${\l}^1 (\mathbb{Z}) $ such that 
$f \geq 0$, $\sum _{n=- \infty} ^{\infty} f(n) =1$, 
$\norm{ [ \lambda (g), \alpha _f (e_N)]} < \epsilon $ and $\norm{ [ \lambda (h), \alpha _f (e_N)]} < \epsilon $, 
where $[x,y] := xy-yx$. 
We will check it. 
For $g,h \in {\l}^1 (\mathbb{Z})$, 
we take a natural number $L$ such that 
\[ 
\max \{ \sum _{n= - \infty }^{-L-1} |g(n)| + \sum _{n= L+1 }^{\infty } |g(n)| , \sum _{n= - \infty }^{-L-1} |h(n)| + \sum _{n= L+1 }^{\infty } |h(n)| \} < \epsilon /4 . 
\] 
Set $R:= \max  \{ |g(-L)| , |g(-L+1)| , \cdots , |g(L)| , |h(-L)| , \cdots , |h(L)| \} $ 
and choose a natural number $K$ such that $K > \max \{ 1/ \epsilon , R \} $. 
We define \[
f(n):=\left\{
\begin{array}{ll}
\frac{1}{4L(2L+1)K^2} & \: ( 1 \leq n \leq 4L(2L+1)K^2 ) \\
0 & \: ( {\rm otherwise} ) \\
\end{array}
\right. . \] 
Then we have 
\begin{align*} 
|g(n)| \sum _{m= - \infty }^{\infty } |f(m-n)-f(m)| 
&= 2|n||g(n)| \frac{1}{4L(2L+1)K^2 } \\ 
& \leq 2LR \frac{1}{4L(2L+1)K^2 } \\ 
&< \frac{\epsilon }{2(2L+1)} \hspace{3em} (-L \leq n \leq L)  
\end{align*} 
and $\sum_{m=- \infty}^{\infty} | f(m-n) - f(m)| \leq 2$ for any $n \in \mathbb{Z} $, 
whence 
\begin{align*} 
\norm{ [ \lambda (g), \alpha _f (e_N)]} 
& \leq \sum _{n=- \infty}^{\infty} |g(n)| \norm{ \alpha ^n (\alpha _f (e_N)) - \alpha _f (e_N) } \\ 
& \leq \sum _{n=- \infty}^{\infty} |g(n)| \sum_{m=- \infty}^{\infty} | f(m-n) - f(m)| \\ 
&< (2L+1) \cdot \frac{\epsilon }{2(2L+1)} + 2 \cdot \frac{\epsilon }{4} 
= \epsilon . 
\end{align*} 
Similarly it follows that $\norm{ [ \lambda (h), \alpha _f (e_N)]} < \epsilon $. 
Thus, for $g,h$ in $C^* (\mathbb{Z})$, we have 
\begin{align*} 
\rho (gh) 
&= \lim_{N \to \infty} \norm{\alpha _f (e_N)^2 \lambda (g) \lambda (h) \alpha _f (e_N)^2} \\ 
& \leq \lim_{N \to \infty} \norm{\alpha _f (e_N) \lambda (g) \alpha _f (e_N)^2 \lambda (h) \alpha _f (e_N)} + \epsilon (\norm{g} + \norm{h} ) \\ 
& \leq \rho (g) \rho (h) + \epsilon (\norm{g} + \norm{h} ) , 
\end{align*} 
whence $\rho (gh) \leq \rho (g) \rho (h)$. 
It also follows that $\rho (g^* g) = \rho (g)^2 $ for $g$ in $C^* (\mathbb{Z})$ since 
\begin{align*} 
& \norm{\alpha _f (e_N)^2 \lambda (g^* ) \lambda (g) \alpha _f (e_N)^2} \\ 
& \qquad \leq \norm{\alpha _f (e_N) \lambda (g^* ) \alpha _f (e_N)^2 \lambda (g) \alpha _f (e_N)} + 2 \epsilon \norm{g} \\ 
& \qquad = \norm{\alpha _f (e_N) \lambda (g) \alpha _f (e_N)}^2 + 2 \epsilon \norm{g} \\ 
& \qquad \leq \norm{\alpha _f (e_N)^2 \lambda (g^* ) \lambda (g) \alpha _f (e_N)^2} + 4 \epsilon \norm{g} 
\end{align*} 
for any $\epsilon >0$ and $f$ in ${\l}^1 (\mathbb{Z}) $ such that 
$f \geq 0$, $\sum _{n=- \infty} ^{\infty} f(n) =1$ and $\norm{ [ \lambda (g), \alpha _f (e_N)]} < \epsilon $. 
So we can conclude that $\rho$ is a $C^*$-semi-norm.

We will check that $\rho$ is non-degenerate. 
At first, since 
\begin{align*} 
\hat{g} (t) 
&= \langle \sum _n g(n) e^{int} \xi _0 , \xi _0 \rangle 
= \langle \sum _n g(n) e^{int} U^n \pi (e_N ) \xi _0 , \pi (e_N ) \xi _0 \rangle \\ 
&= \langle (\pi \rtimes U) (e_N \hat{\alpha _t} (\lambda (g)) e_N) \xi _0 , \xi_0 \rangle 
\end{align*} 
for $g \in {\l}^1 (\mathbb{Z})$ 
and ${\l}^1 (\mathbb{Z})$ is dense in $C^* (\mathbb{Z})$, 
where $\hat{g}$ is the Fourier transform of $g$, 
the same equality holds for any $g \in C^* (\mathbb{Z})$. 
Suppose that $\rho (g) =0$ for $g \in C^* (\mathbb{Z})$. 
We may assume that $g \geq 0$. 
Since 
\begin{align*} 
\hat{g} (t) 
&= \langle (\pi \rtimes U) (e_N \hat{\alpha _t} (\lambda (g)) e_N) \xi _0 , \xi_0 \rangle \\ 
& \leq \norm{(\pi \rtimes U)(e_N \hat{\alpha _t} (\lambda (g)) e_N)} 
= \norm{e_N \hat{\alpha _t} (\lambda (g)) e_N} \\ 
&= \norm{\hat{\alpha _t} (e_N  \lambda (g) e_N)} 
= \norm{e_N  \lambda (g) e_N} \\ 
& \rightarrow \rho (g) =0 
\end{align*} 
for any $t \in \mathbb{R} / 2 \pi \mathbb{Z} $, 
it follows that $g=0$. 
Thus $\rho$ is a $C^*$-norm.

Let $h$ be an element of $C^* (\mathbb{Z}) $ 
such that $\hat{h} \geq 0$, $\norm{\hat{h}} =1$ and $\supp \hat{h} \subset (q- \epsilon , q+ \epsilon )$. 
Then we have 
\begin{align*} 
\norm{\pi (e_N) E_U (q- \epsilon , q+ \epsilon )}^2 
&= \norm{\pi (e_N) E_U (q- \epsilon , q+ \epsilon ) \pi (e_N) } \\ 
& \geq \norm{\pi (e_N) \hat{h} (U) \pi (e_N) } \\ 
&= \norm{(\pi \rtimes U) (e_N \lambda (h ) e_N ) } 
= \norm{e_N \lambda (h ) e_N } \\ 
& \rightarrow \rho (\lambda (h ) ) 
= \norm{h } = 1. 
\end{align*} 
Now we reach the assertion.
\end{proof}
 
\begin{lemma} \label{aeig}
For any $\epsilon >0$, 
there exists a $\delta > 0$ such that 
whenever $\norm{U \xi - e^{iq} \xi } < \delta $ for a unit vector $\xi$ in ${\mathcal H}$ 
and a $q$ in $\mathbb{R} / 2 \pi \mathbb{Z} \ (\simeq \mathbb{T})$, then 
\[ \norm{E_U (q- \epsilon , q+ \epsilon ) \xi } > 1- \epsilon . \] 
\end{lemma}

\begin{proof}
For a unit vector $\xi$ in ${\mathcal H}$ and a $q$ in $\mathbb{R} / 2 \pi \mathbb{Z}$, 
we define a probability measure $\mu := \mu _{\xi , q} $ on $\mathbb{R} / 2 \pi \mathbb{Z}$ 
by $\mu (S) = \langle E_U (S+q) \xi , \xi \rangle $. 
Since 
\begin{align*} 
\langle U \xi , e^{iq} \xi \rangle 
&= \int_{\mathbb{R} / 2 \pi \mathbb{Z}} e^{i(p-q)} \, d \langle E_U (p) \xi , \xi \rangle \\ 
&= \int_{\mathbb{R} / 2 \pi \mathbb{Z}} e^{ip} \, d \langle E_U (p+q) \xi , \xi \rangle \\ 
&= \int_{\mathbb{R} / 2 \pi \mathbb{Z}} e^{ip} \, d \mu (p) , 
\end{align*} 
it follows that 
\[ 
\norm{U \xi - e^{iq} \xi } ^2 
=2 \int_{\mathbb{R} / 2 \pi \mathbb{Z}} (1- \cos p) \, d \mu (p) . 
\] 
Thus, if $\norm{U \xi - e^{iq} \xi } < \delta $, then 
\[ 1- \int \cos p \, d \mu (p) < \delta ^2 /2 . \] 

Suppose that the assertion is false. 
Then there are an $\epsilon >0$, a sequence $(\xi _m)_m$ of unit vectors in ${\mathcal H}$, 
and a sequence $(q_m)_m$ in $\mathbb{R} / 2 \pi \mathbb{Z}$ such that 
\begin{align*} 
& \lim_{m \to \infty} \norm{U \xi _m - e^{i q_m} \xi_m } = 0, \\ 
& \norm{E_U (q_m - \epsilon , q_m + \epsilon ) \xi _m } \leq 1- \epsilon . 
\end{align*} 
Then, by taking a weak cluster point of $(\mu _{\xi _m , q_m} )_m$ 
(in the dual of $C(\mathbb{R} / 2 \pi \mathbb{Z})$), 
we can find a measure $\mu $ on $\mathbb{R} / 2 \pi \mathbb{Z}$ such that 
\[ 
\mu (\mathbb{R} / 2 \pi \mathbb{Z}) \leq 1, \quad 
\mu (- \epsilon , \epsilon ) \leq (1- \epsilon ) ^2, \quad 
\int \cos p \, d \mu (p) = 1. 
\] 
The first and third conditions imply that $\mu$ is the Dirac measure at $p=0$, 
which contradicts the second condition. 
Thus we have reached the assertion.
\end{proof}
 
\begin{lemma} \label{dini}
If $x \in A$ satisfies $xp=0$, 
where $p$ is the support projection of $\omega = \langle \pi (\cdot ) \xi _0 , \xi _0 \rangle $, 
then it follows that $\norm{xe_N} \rightarrow 0$ as $N \rightarrow \infty $. 
\end{lemma}
 
\begin{proof}
For a state $\phi $ on $A$, 
we define $f_N (\phi ) := \phi (x e_N^2 x^*)$. 
Since $(e_N^2)_N $ is also decreasing, 
$(f_N (\phi ))_N $ converges to $\phi (xpx^*) =0$. 
Since $f_N (\phi )$ is continuous for each $N$ as a function on the state space of $A$ with the weak* topology, 
which is compact, 
it follows that $(f_N )_N $ converges uniformly to $0$. 
Thus we have $\norm{x e_N^2 x^* } = \sup _{\phi } f_N (\phi ) \rightarrow 0$, 
whence $ \norm{x e_N } \rightarrow 0$. 
\end{proof}
 
\begin{lemma} \label{afix}  
Let $x$ be an element of $T$ and $\beta $ an automorphism of $A$ 
and $V$ a unitary such that $V \xi _0 = \xi _0 $ and {\rm Ad}$V \circ \pi = \pi \circ \beta $. 
Then for any $\epsilon > 0$ 
there exists a $b \in T$ 
such that $xb = b$ and $\norm{\beta (b) - b} < \epsilon $. 
\end{lemma}
 
\begin{proof}
At first, note that since $V \xi _0 = \xi _0 $ implies $\omega (\beta (p))=1 $, 
we have $\beta (p) =p $, 
where $p:= \supp \omega $. 
Let $(e_N)_N $ be a decreasing sequence for $\xi _0$ as before. 
Let $f$ be a function on $\mathbb{Z}$ 
such that $f \geq  0, \sum_{n \in \mathbb{Z}} f(n) =1 $, 
and $\sum_{n \in \mathbb{Z}} |f(n-1) - f(n)| < \epsilon $ 
(for example, $f(n) =1/(2N+1) \ {\rm for} \ -N \leq n \leq N , \ =0 \ {\rm otherwise} $), 
and let $b_N = \beta _f (e_N)$. 
Then we have 
\begin{align*} 
& \pi (b_N) \xi _0 
= \xi _0 , \\ 
& \norm{\beta (b_N) - b_N} 
< \epsilon . 
\end{align*} 
Take an element $c \in T$ such that $cx=c$. 
Then it follows that $(c-1)p=0$ 
since $\pi (p)$ is the one-dimensional projection onto $\mathbb{C} \xi _0 $ 
(here $1$ is in the unitization of $A$ when $A$ is non-unital). 
By \lemref{dini}, we have 
\[ \norm{c e_N - e_N} \rightarrow 0. \] 
By \lemref{en}, it follows that 
\[ \norm{cb_N - b_N} \rightarrow 0 . \] 
Let \[
g_N (t):=\left\{
\begin{array}{ll}
\frac{N}{N-1} t & \: ( 0 \leq t \leq 1-1/N ) \\
1 & \: ( 1-1/N \leq t \leq 1 ) 
\end{array}
\right. . \] 
Then it follows that $\sup _{0 \leq t \leq 1 } |g_N (t)-t| \rightarrow 0$. 
Now $b:= g_N (c b_N c)$ for a sufficiently large $N$ satisfies all of the conditions of the lemma.
\end{proof}

\begin{lemma} \label{ind} 
Let $A$ be a separable $C^*$-algebra, $\alpha$ an automorphism on $A$, 
$\pi$ a faithful $\alpha$-covariant irreducible representation of $A$ on a Hilbert space ${\mathcal H}$, 
$U$ the implementing unitary for $\alpha $, 
and $\xi _0$ a unit vector such that $U \xi _0 = \xi _0$. 
Let $(p_1 , p_2 , \cdots , p_m)$ be a sequence in $\mathbb{R} / 2 \pi \mathbb{Z}$ 
and $(x_0 , x_1 , \cdots , x_m)$ a sequence in $A_1 $ with $x_0 \in T$ such that 
\begin{align*} 
& U \pi (x_k) \xi _0 = e^{ip_k} \pi (x_k) \xi _0 , \\ 
& x_j^* x_k = 0 \hspace{3em} 
{\rm if} \ j \neq k , \\ 
& x_j x_k = 0 \hspace{3em} 
{\rm if} \ k \neq 0 , \\ 
& x_j^* x_j x_0 = x_0 \hspace{3em} 
{\rm if} \ j \neq 0 
\end{align*} 
for $j,k=0,1, \cdots ,m$, where $p_0 =0$. 
Let $(q_1 , q_2 , \cdots , q_n)$ be a sequence in $\mathbb{R} / 2 \pi \mathbb{Z}$ and $\epsilon >0$.

Then there exist a sequence $(y_0 , y_1 , \cdots , y_n)$ in $A$ 
with $y_0 \in T$ and $\norm{y_{\l}} = 1 $ for $\l=0,1,\cdots , n$, 
and $v$ in ${\mathcal U} (A)$ 
such that $\norm{v-1} < \epsilon $, 
\begin{align*} 
& x_0 y_{\l} = y_{\l} x_0 = y_{\l} , \\ 
& y_j^* y_{\l} = 0 \hspace{3em} 
{\rm if} \ j \neq \l , \\ 
& y_j y_{\l} = 0 \hspace{3em} 
{\rm if} \ \l \neq 0 , \\ 
& y_j^* y_j y_0 = y_0 \hspace{3em} 
{\rm if} \ j \neq 0 
\end{align*} 
for $j,\l=0,1, \cdots ,n$ and 
\begin{align*} 
& U^{(v)} \pi (x_k y_{\l}) \xi _0 = e^{i(p_k + q_{\l} )} \pi (x_k y_{\l}) \xi _0 , \\ 
& \norm{(\alpha ^{(v)} (x_k y_{\l}) - e^{i(p_k + q_{\l} )} x_k y_{\l} ) y_0 } < \epsilon , 
\end{align*} 
for $k=0,1, \cdots , m$ and $\l=0,1, \cdots , n $ with $q_0 =0$. 
\end{lemma}
 
\begin{proof}
We may assume that $(q_{\l} - \epsilon , q_{\l} + \epsilon )$, $\l=1, \cdots , n$ are identical or mutually disjoint. 
Let $(e_N)_N $ be a decreasing sequence in $T$ associated with $\xi _0$ as before. 
We may suppose that $e_1=x_0$. 
By \lemref{dini}, 
we can take a sufficiently large number $N$ such that 
\[ \norm{(\alpha (x_k) - e^{ip_k} x_k ) e_N} < \epsilon \] 
for $k=0,1, \cdots , m$. 
Let $P$ be the spectral projection of $\pi (e_N ) $ corresponding to the eigenvalue $1$. 
Then we have $\pi (e_N ) P=P$ and $\pi (e_{N+1} ) P = \pi (e_{N+1} )$. 
By \lemref{eigqq}, it follows for $\l=1, \cdots , n$ that 
$\norm{\pi (e_{N+1} ) E_U (q_{\l} - \epsilon , q_{\l} + \epsilon  ) \pi (e_{N+1} ) } =1$. 
Since $\pi (e_{N+1})P= \pi (e_{N+1}) $, 
we have $\norm{ PE_U (q_{\l} - \epsilon  , q_{\l} + \epsilon  )P} =1$ for $\l=1, \cdots , n$. 
So there is a unit vector $\eta _{\l} $ in $P {\mathcal H}$ for each $\l=1, \cdots , n$ such that 
$1- \langle E_U (q_{\l} - \epsilon  , q_{\l} + \epsilon  ) \eta _{\l} , \eta _{\l} \rangle < {\epsilon } ^2 $, 
which is equivalent to 
\[ \norm{E_U (q_{\l} - \epsilon  , q_{\l} + \epsilon  ) \eta _{\l} - \eta _{\l} } < \epsilon  . \] 
Since $E_U (q_{\l} - \epsilon , q_{\l} + \epsilon ) E_U (q_k - \epsilon , q_k + \epsilon ) =0$ for $q_{\l} \neq q_k $, 
we have $ |\langle \eta _{\l} , \eta _k \rangle | < 2\epsilon  $. 
For $q_{\l_1} = \cdots =q_{\l_r} (=q_{\l} )$, 
we want to take a mutually orthogonal family $(\eta _{\l_1} , \cdots , \eta _{\l_r} )$ 
such that $\eta _{\l_j} \in P {\mathcal H} $ 
and $1- \langle E_U (q_{\l} - \epsilon  , q_{\l} + \epsilon  ) \eta _{\l_j} , \eta _{\l_j} \rangle < {\epsilon } ^2 $ 
for $j=1, \cdots , r$. 
By \lemref{eigqq}, it follows that 
\begin{align*} 
& \norm{\pi (e_{N+1} ) E_U (q_{\l} - \frac{\epsilon }{2} , q_{\l} + \frac{\epsilon }{2} ) \pi (e_{N+1} ) } \\ 
& \qquad = \norm{\pi (e_{N+1} ) E_U (q_{\l} - \epsilon , q_{\l} + \epsilon  ) \pi (e_{N+1} ) } 
=1. 
\end{align*} 
Let $h:\mathbb{R} /2\pi \mathbb{Z}  \rightarrow [0,1] $ be a continuous function 
such that $h=1$ on $(q_{\l} - \epsilon /2 , q_{\l} + \epsilon /2 )$ 
and $h=0$ on the complement of $(q_{\l} - \epsilon , q_{\l} + \epsilon  )$. 
Then it follows that 
\begin{align*} 
& \pi (e_{N+1} ) E_U (q_{\l} - \epsilon , q_{\l} + \epsilon  ) \pi (e_{N+1} ) \\ 
& \qquad \geq \pi (e_{N+1} ) h(U) \pi (e_{N+1} ) \\ 
& \qquad \geq \pi (e_{N+1} ) E_U (q_{\l} - \epsilon /2 , q_{\l} + \epsilon /2 ) \pi (e_{N+1} ) , 
\end{align*} 
which implies that $\norm{\pi (e_{N+1} ) h(U) \pi (e_{N+1} ) } =1$. 
By \lemref{ncpt}, we have 
\[ \norm{Q(\pi (e_{N+1} ) h(U) \pi (e_{N+1} ))} =1 , \] 
where $Q: {\bf B} ({\mathcal H}) \rightarrow {\bf B} ({\mathcal H})/ {\bf K} ({\mathcal H}) $ is the quotient map. 
Hence it follows that 
\begin{align*} 
\norm{PE_U (q_{\l} - \epsilon  , q_{\l} + \epsilon  )P+K } 
& \geq \norm{Q(PE_U (q_{\l} - \epsilon  , q_{\l} + \epsilon  )P)} \\ 
& \geq \norm{Q(\pi (e_{N+1} ) E_U (q_{\l} - \epsilon , q_{\l} + \epsilon  ) \pi (e_{N+1} ))} \\ 
& \geq \norm{Q(\pi (e_{N+1} ) h(U) \pi (e_{N+1} ))} \\ 
&= 1 
\end{align*} 
for any $K \in {\bf K} ({\mathcal H}) $. 
For a finite rank projection $F$ such that $F \leq P $, let 

\[ K:= -(F E_U (q_{\l} - \epsilon , q_{\l} + \epsilon  ) P+P E_U (q_{\l} - \epsilon , q_{\l} + \epsilon  ) F-F E_U (q_{\l} - \epsilon , q_{\l} + \epsilon  ) F) .\] 
Then it follows that $K$ is a finite rank operator, and hence 
\begin{align*} 
& \norm{(P-F) E_U (q_{\l} - \epsilon  , q_{\l} + \epsilon  ) (P-F)} \\ 
& \qquad = \norm{PE_U (q_{\l} - \epsilon  , q_{\l} + \epsilon  )P+K} 
=1 . 
\end{align*} 
Thus we can take a desired family $(\eta _{\l_1} , \cdots , \eta _{\l_r} )$ inductively. 
We use the Gram-Schmidt orthogonalization for all $\eta _{\l} $'s. 
By \lemref{gs}, we have 
$\norm{E_U (q_{\l} - \epsilon  , q_{\l} + \epsilon  ) \eta _{\l} - \eta _{\l} } < r_n \epsilon $ after this process, 
where $r_n$ is a positive real number dependent on $n$.  

By Kadison's transitivity, 
there exists a $y_{\l} $ in $A$ such that $\norm{y_{\l} } =1$ and 
\[ \pi (y_{\l}) \xi _0 = \eta _{\l}  \] 
for $\l=1,2, \cdots , n$. 
There also exists a $b$ in $A_{+} $ such that 
\[ \pi (b) \eta _{\l} =(\l+1) \eta _{\l} \] 
for $\l=0,1, \cdots , n$, 
where $\eta _0 = \xi _0$. 
Since $\pi (e_N ) P = P$, 
we may replace $b$ by $e_N b e_N $, 
and hence we may assume that $x_0 b=b$. 
Let $(f_0 , f_1 , \cdots , f_n ) $ be a sequence of non-negative functions in $C_0 ( 0 , \infty ) $ with norm $1$ 
such that $f_{\l} (\l+1) =1$ and supp$(f_{\l} ) \subset (\l+ 1/2 , \l+ 3/2)$ for $\l=0,1, \cdots , n$. 
Then, since 
\[ \pi (f_{\l} (b) y_{\l} f_0 (b)) \eta _0 = \eta _{\l} , \] 
we may replace $y_{\l} $ by $f_{\l} (b) y_{\l} f_0 (b)$ for $\l=0,1, \cdots , n$. 
Then it follows that $x_0 y_{\l} = y_{\l} x_0 = y_{\l}$, $y_j^* y_{\l} =0$ for $j \neq \l$, 
and $y_j y_{\l} =0 $ for $j,\l=1, 2, \cdots , n$ 
besides the original conditions $\pi (y_{\l} ) \xi _0 = \eta _{\l} $ and $\norm{y_{\l}} = 1$.
 
Since $\langle \pi (y_1^* y_1 ) \xi _0 , \xi _0 \rangle = \norm{\eta _1 } = 1$ 
and $\norm{\pi (y_1^* y_1) \xi _0 } \leq 1$ 
(because $\norm{y_1} = \norm{\xi _0} = 1$), 
it follows that $\pi (y_1^* y_1 ) \xi _0 = \xi_0 $. 
Let $f$ be a non-negative function in $C_0 (0 , \infty )$ 
such that $f(t)=t^{-1/2} $ around $t=1$ and $t f(t)^2 \leq 1$ for all $t > 0$. 
Then we have $\pi (f(y_1^* y_1 )) \xi _0 = \xi _0$, and so 
\[ \pi (y_1 f(y_1^* y_1 )) \xi _0 = \eta _1 . \] 
Replacing $y_1 $ by $y_1 f(y_1^* y_1 )$, 
it follows that $y_1^* y_1 \in T$, 
since $tf(t)^2 \equiv 1$ around $t=1$. 
Take a $z_1 \in T$ such that $y_1^* y_1 z_1 = z_1 $. 
Replacing $y_2$ by $y_2 z_1 f(z_1 y_2^* y_2 z_1 )$, 
it follows that $y_2 y_1^* y_1 = y_2 $ and $y_2^* y_2 \in T$. 
Take a $z_2 \in T$ such that $y_2^* y_2 z_2 = z_2 $. 
Inductively, we replace $y_i $ by $y_i z_{i-1} f(z_{i-1} y_i^* y_i z_{i-1} )$ and obtain a $z_i \in T$. 
Set $y_0 := z_n $. 
Then we have $y_0 y_{\l}^* y_{\l} = y_0 $ and $y_0 y_{\l} =0$ for $\l=1,2, \cdots , n$. 
Thus $(y_0 , \cdots , y_n )$ satisfies the first four conditions.

Since 
\begin{align*} 
\norm{U \eta _{\l} - e^{iq_{\l}} \eta _{\l} } 
& \leq 2 \norm{E_U (q_{\l} - \epsilon , q_{\l} + \epsilon ) \eta _{\l} - \eta _{\l} } \\ 
& \qquad + \norm{U E_U (q_{\l} - \epsilon , q_{\l} + \epsilon ) \eta _{\l} - e^{iq_{\l}} E_U (q_{\l} - \epsilon , q_{\l} + \epsilon ) \eta _{\l} } \\ 
&< 2 r_n  \epsilon + \norm{\int _{q_{\l} - \epsilon }^{q_{\l} + \epsilon } (e^{it} - e^{iq_{\l} } ) dE_U (t) \eta _{\l} } \\ 
& \leq (2r_n +2) \epsilon , 
\end{align*} 
it follows that 
\begin{align*} 
& \quad \norm{U \pi (x_k y_{\l} ) \xi _0 - e^{i(p_k + q_{\l} )} \pi (x_k y_{\l} ) \xi _0 } \\ 
& \leq \norm{\pi (\alpha (x_k )) U \eta _{\l} - e^{iq_{\l} } \pi (\alpha (x_k )) \eta _{\l} } + \norm{\pi (\alpha (x_k )) \eta _{\l} - e^{ip_k } \pi (x_k ) \eta _{\l} } \\ 
& \leq \norm{U \eta _{\l} - e^{iq_{\l} } \eta _{\l} } + \norm{(\alpha (x_k ) - e^{ip_k } x_k ) e_N } \\ 
&< (2r_n +3) \epsilon . 
\end{align*} 
Since the $(m+1)(n+1)$ unit vectors $\pi (x_k y_{\l} ) \xi _0$, $k=0, \cdots , m$, $\l=0, \cdots , n$ 
are mutually orthogonal, and 
\[ U\pi (x_k y_0 ) \xi _0 = e^{ip_k} \pi (x_k y_0 ) \xi _0 \] 
for $k=0, \cdots , m$ and 
\[ \norm{U\pi (x_k y_{\l} ) \xi _0 - e^{i(p_k + q_{\l} )} \pi (x_k y_{\l} ) \xi _0 } < (2r_n +3) \epsilon \] 
for $k=0, \cdots , m , \ \l=1, \cdots , n$, 
we can use \lemref{kad} for a unitary $V$ 
such that $V \pi (x_k y_l ) \xi _0 := e^{i(p_k + q_l )} U^* \pi (x_k y_l ) \xi _0$ 
and $\pi (x_k y_{\l} ) \xi _0 $, $k=0, \cdots , m$ and $\l=0, \cdots , n$ 
to obtain a $v \in {\mathcal U} (A)$ as required except for the last condition. 
Since $y_0 \geq p$, there is another decreasing sequence $(e_N ' )_N $ 
such that $e_1 ' = y_0 $ and $e_N ' \searrow p$. 
By \lemref{dini}, there is a sufficiently large number $N$ such that 
\[ \norm{(\alpha ^{(v)} (x_k y_{\l}) - e^{i(p_k + q_{\l} )} x_k y_{\l} ) e_N ' } < \epsilon . \] 
We replace $y_0 $ by $e_N ' $ and end the proof. 
\end{proof}

\noindent 
\begin{proof}[Proof of \thmref{main}]

Up to conjugacy, we may assume that $\gamma $ is of the form 
\[ \gamma = \bigotimes_{n=1}^{\infty} {\rm Ad \ diag} (1, e^{ip_{n1}} , \cdots , e^{ip_{n,k_n -1}} ) \] 
on $D =  \otimes _{n=1} ^{\infty} M_{k_n} $, 
where diag$(\lambda _1 , \cdots , \lambda _k )$ means the diagonal matrix whose $(i,i)$ component is $\lambda _i $. 
We define $p_{n0} :=0$ for $n=0,1, \cdots $.
 
We have fixed a unit vector $\xi _0 \in {\mathcal H}$ such that $U \xi _0 = \xi _0 $. 
We choose an $e \in T$. 
Let $( \mu _n )$ be a strictly decreasing sequence of positive numbers such that 
\[ nk_1 k_2 \cdots k_n \mu _n < 1 \] 
and let $\epsilon _n := \mu _n - \mu _{n+1} $.
 
Using \lemref{ind} inductively, 
we will find suitable elements $v_m \in {\mathcal U} (A) $ for $m=0,1, \cdots $, 
and $x_{mj} \in A$ for $m=0,1, \cdots $ and $0 \leq j < k_m$. 
When $m=0$ (note that we can set $k_0 :=1$), we define $v_0 :=1$ and $x_{00} := e$.
 
Suppose $v_n \in {\mathcal U} (A) $ and $x_{nj} \in A$ for $0 \leq j < k_n$ are already defined for $n \leq m$ 
so that $x_{n0} \in T$, $\norm{x_{nj} }= 1$, $\norm{v_n -1} < \epsilon _n $, and 
\begin{align*} 
U^{(\overline{v_m} )} \pi (w_i^{(m)} ) \xi _0 &= e^{ip_i^{(m)} } \pi (w_i^{(m)} ) \xi _0 , \\ 
w_j^{(m)*} w_k^{(m)} &= 0 \hspace{3em} {\rm if} \ j \neq k , \\ 
w_j^{(m)} w_k^{(m)} &= 0 \hspace{3em} {\rm if} \ k \neq 0 , \\ 
w_j^{(m)*} w_j^{(m)} w_0^{(m)} &= w_0^{(m)} \hspace{3em} {\rm if} \ j \neq 0 , 
\end{align*} 
for all $i,j,k \in X_m $ and $0= (0,0, \cdots , 0) \in X_m $, where 
\begin{align*} 
\overline{v_m } &:= v_1 v_2 \cdots v_m  , \\ 
X_m &:= \{ i=(i_1 , i_2 , \cdots , i_m ) | 0 \leq i_n < k_n \} , \\ 
w_i^{(m)} &:= x_{1i_1} x_{2i_2 } \cdots x_{mi_m } \hspace{3em} {\rm for} \ i=(i_1 , i_2 , \cdots , i_m ) \in X_m , \\ 
p_i^{(m)} &:= p_{1i_1 } + p_{2i_2 } + \cdots + p_{mi_m } \hspace{3em} {\rm for} \ i=(i_1 , i_2 , \cdots , i_m ) \in X_m , 
\end{align*} 
and $\overline{v_0} :=1$, $X_0 := \{ 0 \} $, $w_0^{(0)} :=e$ and $p_0^{(0)} :=0$ for $m=0$. 
Then, there exist $x_{m+1 , \l} \in A $ for $0 \leq \l < k_{m+1} $ and $v_{m+1 } \in {\mathcal U} (A)$ 
such that $x_{m+1 , 0 } \in T$, $\norm{x_{m+1 , j} } =1$, $\norm{v_{m+1 } -1} < \epsilon _{m+1 } $, and 
\begin{align*} 
w_0^{(m)} x_{m+1 , \l} = x_{m+1,\l} w_0^{(m)} &= x_{m+1,\l} , \\ 
x_{m+1,j}^* x_{m+1,\l} &= 0 \hspace{3em} {\rm if} \ j \neq \l , \\ 
x_{m+1,j} x_{m+1,\l} &= 0 \hspace{3em} {\rm if} \ \l \neq 0 , \\ 
x_{m+1,j}^* x_{m+1,j} x_{m+1,0} &= x_{m+1,0} \hspace{3em} {\rm if} \ j \neq 0
\end{align*} 
for all $j,\l=0,1, \cdots , k_{m+1} -1$ and 
\begin{align*} 
& U^{(\overline{v_{m+1}})} \pi (w_k^{(m)} x_{m+1,\l} ) \xi _0 = e^{i(p_k^{(m)} + p_{m+1,\l} ) } \pi ( w_k^{(m)} x_{m+1,\l} ) \xi _0 , \\ 
& \norm{(\alpha ^{(\overline{v_{m+1}})} (w_k^{(m)} x_{m+1,\l} ) - e^{i(p_k^{(m)} + p_{m+1,\l})} w_k^{(m)} x_{m+1,\l} ) x_{m+1,0} } < \epsilon _{m+1} 
\end{align*} 
for $k \in X_m $ and $\l=0,1, \cdots , k_{m+1} -1$. 
Since $w_k^{(m)} x_{m+1,\l} = w_{(k,\l)}^{(m+1)} $, where $(k,\l) \in X_{m+1} $, it follows that 
\begin{align*} 
w_j^{(m+1)*} w_k^{(m+1)} &= 0 \hspace{3em} {\rm if} \ j \neq k , \\ 
w_j^{(m+1)} w_k^{(m+1)} &= 0 \hspace{3em} {\rm if} \ k \neq 0 , \\ 
w_j^{(m+1)*} w_j^{(m+1)} w_0^{(m+1)} &= w_0^{(m+1)} \hspace{3em} {\rm if} \ j \neq 0 , \\ 
U^{(\overline{v_{m+1} } )} \pi (w_i^{(m+1)} ) \xi _0 &= e^{ip_i^{(m+1)} } \pi (w_i^{(m+1)} ) \xi _0 , \\ 
\norm{(\alpha ^{(\overline{v_{m+1}})} (w_k^{(m+1)} ) - e^{ip_k^{(m+1)} } w_k^{(m+1)} ) w_0^{(m+1)} } &< \epsilon _{m+1} , 
\end{align*} 
for $i,j,k \in X_{m+1} $, where we used $w_0^{(m)} x_{m+1 , \l} = x_{m+1,\l} w_0^{(m)} = x_{m+1,\l} $. 
Since $w_0^{(n-1)} x_{n0} = x_{n0} w_0^{(n-1)} = x_{n0} $ for any $n$ implies $w_0^{(n)} = x_{n0} $, 
we have $(x_{nj} )_{n=1,2,\cdots , 0 \leq j < k_n } $ is a quasi-matrix system (\cite{Ped}, 6.6.1); 
i.e. $(x_{nj})_{n,j} $ satisfies for any $n$, 
\begin{align*} 
x_{n0} & \geq 0, \hspace{3em} \norm{x_{nj}} = 1 \hspace{3em} {\rm for} \ 0 \leq j < k_n , \\  
x_{ni}^* x_{nj} &= 0 \hspace{3em} {\rm for} \ i \neq j , \\ 
x_{ni} x_{nj} &= 0 \hspace{3em} {\rm for} \ j \neq 0 , \\ 
x_{nj}^* x_{nj} x_{n0} &= x_{n0} \hspace{3em} {\rm for} \ 1 \leq j < k_n , \\ 
x_{n0} x_{n+1 , j} &= x_{n+1,j } x_{n0} = x_{n+1,j} \hspace{3em} {\rm for} \ 0 \leq j < k_{n+1} . 
\end{align*} 

We define 
\[ \overline{v} := \lim _m \overline{v_m} = v_1 v_2 \cdots . \] 
Then it follows that $\norm{\overline{v} -1 } < \mu _1 $, and 
\begin{align*} 
& \norm{(\alpha ^{(\overline{v} )} (w_i^{(m)} ) - e^{ip_i^{(m)} } w_i^{(m)} ) w_0^{(m)} } \\ 
& \qquad < \norm{(\alpha ^{(\overline{v_m} )} (w_i^{(m)} ) - e^{ip_i^{(m)} } w_i^{(m)} ) w_0^{(m)} } + 2 \mu _{m+1} \\ 
& \qquad < 2 \mu _m  
\end{align*} 
for $i \in X_m $. 
Note that since $U^{(\overline{v_{m+1}})} \xi _0 = \xi_ 0 $, 
it follows that $U^{(\overline{v})} \xi _0 = \xi_ 0 $, 
which implies $\alpha ^{(\overline{v})} (p)=p$. 

By using the separability of $A$, 
we will impose another condition on the choice of $w_0^{(m)} = x_{m0} $ for each $m$. 
(We only have to replace them for sufficiently large $m$'s.) 
Fix a dense sequence $(a_n )_n $ of $A_{sa} $. 
Let $(e_N)_N$ and $(f_N)_N$ be as in \lemref{dec} and choose $a \in T$ such that 
$w_0^{(m)} a=a$. 
Set $y':= \sum 2^{-N} a e_N a$ and $z_N := f_N (y')$. 
Then we have $w_0^{(m)} z_N =z_N $ for all $N$. 
Let $b$ be an element in $A$. 
Since $z_N \searrow p$, $(z_N (b- \omega (b)) z_N )$ converges $\sigma $-weakly to $p(b- \omega (b))p$, 
which is equal to $0$ since $\pi (p)$ is the $1$-dimensional projection supporting $\omega $. 
So the norm closure of the convex hull of $\{ z_N (b- \omega (b)) z_N \} $ contains $0$. 
Thus for each $\delta >0$ there are positive numbers $(t_i)_i $ with $\sum t_i =1$ 
such that $\norm{\sum t_i z_{N_i} (b- \omega (b)) z_{N_i} } < \delta $. 
Hence whenever $N \geq N_i $ for all $i$, it follows that 
\[ \norm{z_N (b- \omega (b)) z_N } \leq \norm{z_N} \norm{\sum t_i z_{N_i} (b- \omega (b)) z_{N_i} } \norm{z_N} < \delta . \] 
We take such an $N$ and set $\tilde{w_0}^{(m)} := z_N $. 
Applying \lemref{afix}, we may assume that \[ \norm{\alpha^{(\overline{v})} (\tilde{w_0}^{(m)}) - \tilde{w_0}^{(m)} } < \mu _m. \] 
Set $\tilde{a_m} := \sum_{i,j \in X_m } \omega (w_j^{(m)*} a_m w_i^{(m)} ) w_j^{(m)} (\tilde{w}_0^{(m)})^2 w_i^{(m)*} \in B$. 
Then, by setting $b= w_i^{(m)*} a_m w_j^{(m)} $ and $\delta = \mu _m / (k_1 k_2 \cdots k_m ) $ in the argument above, 
we have 
\begin{align*} 
& \norm{(\sum _i w_i^{(m)} (\tilde{w}_0^{(m)})^2 w_i^{(m)*} )(a_m - \tilde{a_m} )(\sum _j w_j^{(m)} (\tilde{w}_0^{(m)})^2 w_j^{(m)*} )} \\ 
& \quad = \norm{\sum _{i,j} w_i^{(m)} (\tilde{w}_0^{(m)})^2 (w_i^{(m)*} a_m w_j^{(m)} - \omega (w_i^{(m)*} a_m w_j^{(m)} ) ) (\tilde{w}_0^{(m)})^2 w_j^{(m)*} } \\ 
& \quad < 1/m . 
\end{align*} 
From now on we just write $w_0^{(m)} $ instead of $\tilde{w_0}^{(m)}$. 

 Let $p_m $ be the spectral projection of $w_0^{(m)} =x_{m0} $ corresponding to $1$. 
We define 
\[ q_m := \sum _{i \in X_m } w_i^{(m)} p_m w_i^{(m)*} \] 
and for $1 \leq n \leq m$, 
\[ q_{mn} := \sum _{i \in X_{n,m } } w_i^{(m)} p_m w_i^{(m)*} \] 
where $X_{n,m } := \{ i \in X_m | i_1 =0, i_2 =0 , \cdots , i_n =0 \} $. 
Then we can prove the same consequences as 6.6 in \cite{Ped}; 
it follows that $q_m , q_{mn} $ with $1 \leq n \leq m$ are projections in $A^{**} $ satisfying 
\begin{align*} 
& x_{ni} q_m = x_{ni} q_{mn} = q_m x_{ni} , \\ 
& x_{ni}^* x_{ni} q_m = x_{n0} q_m = q_{mn} , \\ 
& x_{n+1,0} q_{m'} + \sum _{i=1}^{k_{n+1} -1} x_{n+1,i } x_{n+1, i}^* q_{m'} = x_{n0} q_{m'} 
\end{align*} 
for each $1 \leq n \leq m$, $0 \leq i < k_n $ and $m' >n+1$. 
We will check the last equality. 
Since $x_{n0} q_{n+1} =q_{n+1} x_{n0} = x_{n0} q_{n+1,n} = q_{n+1,n} x_{n0} =q_{n+1,n} $, 
it follows that $p_n q_{n+1} = q_{n+1,n} $ for each $n$. 
Hence we have 
\begin{align*} 
x_{n0} q_{n+2} 
&= x_{n0} q_{n+1} q_{n+2} = q_{n+1,n} q_{n+2} \\ 
&= \sum _{j=0 }^{k_{n+1} -1} x_{n+1,j} p_{n+1} x_{n+1,j}^* q_{n+2} 
= \sum _{j=0 }^{k_{n+1} -1} x_{n+1,j} p_{n+1} q_{n+2} x_{n+1,j}^* \\ 
&= \sum _{j=0 }^{k_{n+1} -1} x_{n+1,j} q_{n+2,n+1} x_{n+1,j}^* 
= \sum _{j=0 }^{k_{n+1} -1} x_{n+1,j} q_{n+2} x_{n+1,j}^* \\ 
&= x_{n+1,0} q_{n+2} + \sum _{i=1}^{k_{n+1} -1} x_{n+1,i } x_{n+1, i}^* q_{n+2} . 
\end{align*} 
Multiplying $q_{m'} $ ($m' >n+1$) by the right side, we get the desired equality.

We define 
\[ r_m := \sum _{i \in X_m } w_i^{(m)} w_0^{(m)^2} w_i^{(m)*} \in A . \] 
Then it follows that $q_m \leq r_m \leq q_{m-1} $. 
Let $q:= $weak*-$\lim q_m $. 
Since $(q_m )_m $ is a decreasing sequence, $q$ is a closed projection in $A^{**} $.  
 
For $i \in X_m $, we have that 
\begin{align*} 
& \hspace{1.1em} \norm{\alpha ^{(\overline{v} )} (w_i^{(m)} w_0^{(m)^2} w_i^{(m)*} ) - w_i^{(m)} w_0^{(m)^2} w_i^{(m)*} } \\ 
& \leq \norm{\alpha ^{(\overline{v} )} (w_i^{(m)} ) w_0^{(m)^2} \alpha ^{(\overline{v} )} (w_i^{(m)*} ) - w_i^{(m)} w_0^{(m)^2} w_i^{(m)*} } + 2 \mu _m \\ 
& \leq 2 \norm{(\alpha ^{(\overline{v} )} (w_i^{(m)} ) - e^{ip_i^{(m)} } w_i^{(m)} ) w_0^{(m)} } + 2 \mu _m \\ 
& \leq 6 \mu _m , 
\end{align*} 
which implies that 
\[ \norm{\alpha ^{(\overline{v} )} (r_m ) - r_m } < 6 k_1 k_2 \cdots k_m \mu _m < 6/m . \] 
Therefore we have $\alpha ^{(\overline{v} )} (q)=q $.
 
For $n \leq m$, $0 \leq i < k_n $ and $j \in X_m $, 
note that if $j \notin X_{n,m} $, it follows that $x_{ni} w_j^{(m)} =0$, 
otherwise we can write $w_{\l}^{(m)} = x_{ni} w_j^{(m)} $ for some $\l \in X_m$ 
(see the proof of \cite{Ped}, 6.6.4). 
So we have 
\begin{align*} 
& \norm{(\alpha ^{(\overline{v} )} (x_{ni} ) - e^{ip_{ni} } x_{ni} ) w_j^{(m)} w_0^{(m)} } \\ 
& \qquad < \norm{(e^{-ip_j^{(m)} } \alpha ^{(\overline{v} )} (x_{ni} w_j^{(m)} ) - e^{ip_{ni} } x_{ni} w_j^{(m)} ) w_0^{(m)} } +2 \mu _m \\ 
& \qquad = \norm{(\alpha ^{(\overline{v} )} (w_{\l}^{(m)} ) - e^{ip_{\l}^{(m)} } w_{\l}^{(m)} ) w_0^{(m)} } +2 \mu _m \\ 
& \qquad < 4 \mu _m , 
\end{align*} 
whenever $w_{\l}^{(m)} = x_{ni} w_j^{(m)} \neq 0$. 
(This inequality also holds when $x_{ni} w_j^{(m)} =0$.) 
Thus it follows that 
\begin{align*} \norm{(\alpha ^{(\overline{v} )} (x_{ni} ) - e^{ip_{ni} } x_{ni} ) q_m } 
& \leq \sum _{j \in X_m } \norm{(\alpha ^{(\overline{v} )} (x_{ni} ) - e^{ip_{ni} } x_{ni} ) w_j^{(m)} p_m } \\ 
&< 4 k_1 k_2 \cdots k_m \mu _m <4/m , 
\end{align*} 
and hence $\alpha ^{(\overline{v} )} (x_{ni} )q = e^{ip_{ni} } x_{ni} q$. 
  
Let $B$ be the $C^*$-subalgebra of $A$ generated by 
$\{ \alpha ^{(\overline{v} )^m} (x_{ni}) | m \in \mathbb{Z} , n=1,2, \cdots , 0 \leq i < k_n \} $. 
Then it is evident that $B$ is invariant under $\alpha ^{(\overline{v} )} $ and $q \in B^{**} $. 
Since $q x_{ni} = x_{ni} q $ and $q$ is $\alpha ^{(\overline{v} )} $-invariant, 
we have $q \alpha ^{(\overline{v} )} (x_{ni} ) = \alpha ^{(\overline{v} )} (x_{ni} ) q$, 
which implies that $q \in B' $. 
Since $(x_{nj} )_{n=1,2,\cdots , 0 \leq j < k_n } $ is a quasi-matrix system and 
\begin{align*} 
& x_{ni}^* x_{ni} q = x_{n0} q , \\ 
& x_{n+1,0} q + \sum _{i=1}^{k_{n+1} -1} x_{n+1,i } x_{n+1, i}^* q = x_{n0} q , 
\end{align*} 
it follows that $\{ x_{nj} q \} _{n=1,2,\cdots , 0 \leq j < k_n } $ generates a UHF algebra which is isomorphic to $D$ by 
\[ w_i^{(n)} q w_j^{(n)*} \mapsto E_{i_1 j_1 } \otimes \cdots \otimes E_{i_n j_n } \in M_{k_1} \otimes \cdots \otimes M_{k_n} \subset D \] 
for $i,j \in X_n $, where $E_{ij} $ denotes the $(i,j)$-matrix unit of $M_{k_n} $ 
(to avoid mistaking indexes, we call the top of a matrix ''the $0$-th row'' and the left end of one ''the $0$-th column''). 
Since $\alpha ^{(\overline{v} )} (x_{ni} )q = e^{ip_{ni} } x_{ni} q$, 
it follows that $\{ x_{nj} q \} _{n,j} $ generates $Bq$ and $(Bq , ( \alpha ^{(\overline{v})} )^{**} |Bq ) \simeq (D, \gamma ) $. 
And since $\norm{r_m (a_m - \tilde{a_m} ) r_m } < 1/m$ and $r_m q=q $ 
(because $q=qq \geq q r_m q \geq q q q =q$ and $q r_m = r_m q$), 
it follows that $qAq=Bq$. 

Finally we show that $xc(q)=0$ implies $x=0$ for $x \in A$. 
Since $w_i^{(n)*} x_{n0} =0$ for $i \in X_n$ unless $i=0 $, it follows that 
\begin{align*} 
q_n x_{n0} &= \sum _{i \in X_n} w_i^{(n)} p_n w_i^{(n)*} x_{n0} \\ 
&= x_{n0} p_n x_{n0}^* x_{n0} = p_n . 
\end{align*} 
Thus we have 
\[ \pi (q_n ) \xi _0 = \pi (q_n x_{n0} ) \xi _0 = \pi (p_n ) \xi _0 = \xi _0 , \] 
which implies $\omega (q_n )=1$ for each $n$. 
Hence it follows that $\omega (q) = 1$, which is equivalent to $q \geq p$. 
So it suffices to show that $xc(p)=0$ implies $x=0$ for $x \in A$. 
Since $\pi (p) \xi _0 = \xi _0 $ and $c(p) \geq p$, we have $\pi (c(p)) \xi _0 = \xi _0 $. 
Thus, for any $x,y \in A$, it follows that 
\[ \pi (x) (\pi (y) \xi _0 ) = \pi (x) \pi (y) \pi (c(p)) \xi _0 = \pi (xc(p)) \pi(y) \xi _0 = 0 . \] 
Note that since $\pi $ is irreducible, we have $\xi _0$ is a cyclic vector. 
Hence it follows that $\pi (x) =0$, which implies $x=0$.
\end{proof}
%%%%%%%%%%%%%%%%%%%%%%%%%%%%%%%%%%%%%%%%%%%%%%%%%%%%%%%%%%%%%

\end{document}